\documentclass[11pt]{article}

\usepackage[margin=1in]{geometry}
\usepackage{amsmath,amssymb,amsthm,mathtools}
\usepackage{enumitem}
\usepackage{hyperref}

\hypersetup{
  colorlinks=true,
  linkcolor=blue,
  citecolor=blue,
  urlcolor=blue
}

\newtheorem{theorem}{Theorem}[section]
\newtheorem{proposition}[theorem]{Proposition}
\newtheorem{lemma}[theorem]{Lemma}
\newtheorem{corollary}[theorem]{Corollary}
\newtheorem{remark}[theorem]{Remark}

\newcommand{\Fp}{\mathbb F_p}
\newcommand{\Fps}{\mathbb F_p^\times}
\newcommand{\C}{\mathbb C}

\newcommand{\ip}[2]{\langle #1,#2\rangle}
\newcommand{\norm}[1]{\left\|#1\right\|}

\title{Sharp Conditioning for Matrix Recovery by Finite Affine Orbits}
\author{Dongwei Li\\
School of Mathematics, Hefei University of Technology\\
Hefei 230601, China\\
\texttt{dongweili@hfut.edu.cn}}
\date{}

\begin{document}

\maketitle

\begin{abstract}
Let \(q=p^h\) be an odd prime power, and let the affine group
\(\mathbb F_q\rtimes\mathbb F_q^\times\) act through its canonical
\((q-1)\)-dimensional irreducible representation.  Qualitative matrix
recovery for these rank-one orbits is known.  We determine sharp lower
singular-value bounds for explicit real generating windows.  First, we
compute the exact least singular value for every nonnegative two-level
window over \(\mathbb F_{p^h}\) with \(h\geq2\), and identify the unique
optimizer when \(q-1\geq10\).  The resulting
conditioning stays bounded away from zero on every fixed
odd-characteristic tower and is within an explicit characteristic-dependent
factor of the best possible value over all real windows.

For every \(q=3^h\), \(h\geq2\), we construct a real three-level
absolute-trace window (constant on the fibers of
\(\operatorname{Tr}_{\mathbb F_q/\mathbb F_3}\)) whose least singular value is
\[
 \frac{q(\sqrt2-1)}{q(2-\sqrt2)-1}>\frac1{\sqrt2}.
\]
For the full class of real windows constant on the three trace classes, we
reduce the least singular value to four scalar expressions and a symmetric
\(2\times2\) matrix.  This yields the exact global optimum and all equality
cases: the displayed trace window is uniquely optimal up to global sign and
interchange of the two nonzero trace classes.  Thus zero trace mean follows
from optimality.  The proof combines an explicit sum-of-squares identity,
uniform quadratic-form certificates, and a finite-geometric block
decomposition of the orbit measurement operator.
\end{abstract}



\section{Introduction}

Phase retrieval asks for recovery of a vector from magnitudes of frame
coefficients; its finite-dimensional injectivity and stability theories have
developed from the foundational frame formulation in
\cite{BalanCasazzaEdidin2006} through algebraic and quantitative criteria in
\cite{BandeiraCahillMixonNelson2014,ConcaEdidinHeringVinzant2015}.
A stronger linear problem is matrix recovery, also called the
maximal span property: a family \((\psi_i)_{i\in I}\subset \C^d\) does matrix
recovery if the rank-one projectors \(\psi_i\psi_i^*\) span the real or
complex matrix space under consideration.  Matrix recovery implies phase
retrieval by applying the linear recovery map to \(xx^*\), and is the
linear measurement model underlying informationally complete rank-one
tomography \cite{Scott2006}.

Highly structured examples are important because generic measurements are
abundant but rarely explain how stability depends on the dimension.  Group
orbits are especially natural: one chooses a finite group \(G\), a unitary
representation \(\pi:G\to U(\mathcal H)\), and a generating vector
\(\phi\in\mathcal H\), and studies the orbit
\((\pi(g)\phi)_{g\in G}\); geometrically uniform frames provide an early
systematic instance of this principle \cite{EldarBoelcskei2003}.

This paper concerns the affine group
\[
        G=\Fp\rtimes \Fps,\qquad
        (k,l)(k',l')=(k+lk',ll'),
\]
where \(p\) is an odd prime.  Its canonical irreducible representation on
\(\mathcal H=\C^{\Fps}\) is
\[
        (\pi(k,l)f)(m)=e_p(km)f(lm),
        \qquad
        e_p(x)=\exp(2\pi i x/p).
\]
The dimension is \(d=p-1\), while \(|G|=p(p-1)=d(d+1)\), exceeding the dimension \(d^2\) of the matrix space by \(d\).

Bartusel, F\"uhr and Oussa proved that this representation admits generating
vectors whose affine orbit does matrix recovery.  More precisely, they gave
a complete criterion in terms of two finite matrices: a multiplicative
character vector \(c_\phi\) and a rectangular matrix \(B_\phi\).  Their result
is qualitative and algebraic: \(c_\phi\) must have no zero character
coefficient, and \(B_\phi\) must have full column rank.

The purpose of the present paper is to solve the following quantitative design
problem in a natural explicit family: among normalized vectors taking one value
at a distinguished coordinate and a second value elsewhere, which affine orbit
has the largest least singular value?  This question is not answered by a
qualitative maximal-spanning criterion.

Theorem \ref{thm:ppoptimal} gives the exact least singular value of every
nonnegative two-level generator on every odd-characteristic prime-power tower
of extension degree at least two and
identifies the unique optimizer once \(d\geq10\).  Theorem~\ref{thm:traceflat}
then exposes a different phenomenon: on every \(3\)-power tower, an explicit
three-level absolute-trace window, defined using
\(\operatorname{Tr}_{\mathbb F_q/\mathbb F_3}\), has a least singular value uniformly larger
than \(1/\sqrt2\).  Theorem~\ref{thm:tracespectrum} gives the full spectral
reduction for the zero-sum trace class.  Theorem~\ref{thm:generaltrace}
extends the least-singular-value reduction to arbitrary real three-value
trace windows, without a zero-sum assumption.
Theorem~\ref{thm:traceglobal} proves the exact global optimum over this
entire class and classifies all equality cases.  Its proof uses a
sum-of-squares identity on the principal sign region and six explicit
matrix certificates on its complement.  A convex interpolation between
the \(q=9\) and limiting certificates makes the proof uniform in the
field size.  In particular, zero trace mean is forced by optimality.
Corollary~\ref{cor:traceisolation} supplies a quantitative isolation
estimate within the zero-sum class for an explicit range of errors.
Theorem~\ref{thm:universalupper} also gives a comparison
with every real unit generator: the optimized two-level orbit is within a
\(p\)-dependent constant factor of the globally best real conditioning, and
Proposition~\ref{prop:ceilingrigidity} identifies the simultaneous spectral
equalities that any global extremizer would have to satisfy.

\begin{theorem}[Baseline affine orbit]
\label{thm:zero}
Let \(p\geq 5\) be prime and put \(n=p-1\).  Define
\[
        \phi(1)=0,
        \qquad
        \phi(m)=\frac{1}{\sqrt{n-1}}\quad (m\neq 1).
\]
Then the affine orbit measurement operator
\[
        \mathcal M_\phi:\C^{\Fps\times \Fps}\to \C^G,\qquad
        \mathcal M_\phi(A)(g)
        =
        \ip{A\pi(g)\phi}{\pi(g)\phi},
\]
satisfies
\[
        \norm{\mathcal M_\phi(A)}_{\ell^2(G)}
        \geq
        \frac{2\sqrt p}{n-1}\,
        \sin\left(\frac{\pi}{2n}\right)\norm{A}_F
        \qquad
        \text{for every }A\in \C^{\Fps\times \Fps}.
\]
Consequently \(\mathcal M_\phi\) is injective and admits a linear inverse on
its range with operator norm \(O(d^{3/2})\).
Moreover this order is sharp for this generator: there is a nonzero matrix
\(A\) such that
\[
        \norm{\mathcal M_\phi(A)}_{\ell^2(G)}
        \leq
        \frac{\sqrt{3}\pi\sqrt p}{(n-1)n}\,\norm{A}_F.
\]
\end{theorem}

The constant is explicit:
\[
        \frac{2\sqrt p}{n-1}\,
        \sin\left(\frac{\pi}{2n}\right)
        \geq
        \frac{2\sqrt p}{(n-1)n}
\]
for \(n\geq2\).  The second theorem below improves this baseline by a factor of
order \(d^{1/2}\), while retaining a completely explicit orbit of size
\(d(d+1)\).

\section{Related work and scope}

The exact recovery side is due to Bartusel, F\"uhr and Oussa
\cite{BartuselFuhrOussa2023}.  Their Theorem 6.5 characterizes matrix
recovery for the canonical affine representation over prime fields, their
Corollary 6.6 gives an explicit recovery algorithm, and their Theorem 6.7
gives a concrete qualitative example.

Cheng's Lemma~2.4 already expresses the frame operator on an isotypic
component as a scaled Gram matrix tensored with the identity.  We use this
standard representation-theoretic mechanism; the block reduction itself is
not claimed as a new general principle.  Our contribution is the evaluation
and optimization of the resulting blocks for explicit affine windows.
In particular, the present work gives an exact
singular-value formula, not merely an injectivity test; it identifies the
unique optimum in the nonnegative two-level family; and it compares the
resulting constant with a universal upper benchmark for all real generators.
Necessary algebraic conditions for equality at that benchmark are also given.  In a different direction, a three-level absolute-trace
construction on every \(3\)-power tower has an exact uniform stability
constant and strictly outperforms the nonnegative two-level family on every such tower.
Theorem~\ref{thm:traceglobal} further proves its exact global optimality
among all real three-value trace windows, without imposing zero trace mean.
These statements
are all absent from a qualitative maximal-spanning assertion.

After the appearance of \cite{BartuselFuhrOussa2023}, Cheng proved that the
canonical \((q-1)\)-dimensional representation of \(\mathrm{GA}(1,q)\) admits
maximal spanning vectors, and gave a characterization of such vectors
\cite[Theorem~1.2 and Remark~4.7]{Cheng2025}.  Thus qualitative matrix
recovery over prime-power fields is not claimed here as new.  The present
contribution is the exact two-level optimization, together with the
closed-form \(3\)-power trace family and its global optimization over
all real three-value trace windows.
Cheng's existence and characterization result does not state either of these
quantitative conclusions.

General stability theory for phase retrieval and generalized phase retrieval
is extensive; see, for example, \cite{GrohsKoppensteinerRathmair2020} for a
survey and \cite{Zhuang2019real,Zhuang2019complex} for Lipschitz statements
for generalized phase retrieval.  Those results are not a substitute for an
explicit stability calculation for a specific finite group orbit, since the
constants in the general theory are typically encoded as compactness minima
or frame bounds rather than as closed-form estimates for a deterministic
family.

For broader context, finite tight frames and their erasure/stability questions
are developed in \cite{DuffinSchaeffer1952,BenedettoFickus2003,
CasazzaKovacevic2003,StrohmerHeath2003}; the Welch bound
\cite{Welch1974} remains a basic benchmark for structured frame design.
Standard references for finite-frame constructions are
\cite{CasazzaKutyniok2013,Waldron2018}.  Deterministic full-spark and
erasure-robust frames obtained from totally nonsingular matrices are studied
in \cite{Li2026TNS}.  That linear erasure problem is distinct from the
rank-one orbit measurement problem here, but shares the guiding aim of
obtaining verifiable robustness from explicit algebraic structure.
The tight informationally complete
measurements appearing in Lemma~\ref{lem:ceiling2design} belong to the
projective-design framework of \cite{DelsarteGoethalsSeidel1977,
EldarForney2002,RenesEtAl2004,RoyScott2007}.  We use only elementary
finite-field Fourier and finite-group representation tools; convenient
background sources are \cite{LidlNiederreiter1997,Terras1999,Serre1977}.
Finally, convex lifting and nonconvex algorithmic approaches to generic phase
retrieval \cite{CandesStrohmerVoroninski2013,CandesLiSoltanolkotabi2015}
provide useful complementary motivation, but do not give the deterministic
orbit-specific least-singular-value formula proved here.
Recent work on nilpotent-group orbits likewise establishes broad qualitative
phase-retrieval classes \cite{FuhrOussa2023}; it does not address sharp
conditioning for the affine orbit considered here.

\section{The Affine Block Criterion}

We recall the two objects in the Bartusel--F\"uhr--Oussa criterion.  Throughout
this section \(p\) is an odd prime and \(n=p-1\).  For a real-valued vector
\(\phi\in\mathbb R^{\Fps}\), define
\[
        c_\phi(\chi)
        =
        \sum_{m\in\Fps} |\phi(-m)|^2\chi(m),
        \qquad
        \chi\in\widehat{\Fps},
\]
and
\[
        B_\phi(m,t)=\phi(mt)\phi(m(t+1)),
        \qquad
        m\in\Fps,\quad t\in\Fp\setminus\{0,-1\}.
\]
Thus \(B_\phi\) is an \(n\times(n-1)\) matrix.

\begin{proposition}[Quantitative affine block criterion]
\label{prop:block}
Let \(\phi\in\mathbb R^{\Fps}\) be a unit vector and put
\[
        \alpha_\phi=\min_{\chi\in\widehat{\Fps}} |c_\phi(\chi)|,
        \qquad
        \beta_\phi=\sigma_{\min}(B_\phi).
\]
Then, for every matrix \(A\in \C^{\Fps\times \Fps}\),
\[
        \norm{\mathcal M_\phi(A)}_{\ell^2(G)}
        \geq
        \sqrt p\,\min\{\alpha_\phi,\beta_\phi\}\,\norm{A}_F.
\]
In particular, if \(\alpha_\phi>0\) and \(\beta_\phi>0\), then
\(\pi(G)\phi\) does matrix recovery.
\end{proposition}

\begin{proof}
We record the quantitative content of the decomposition in
\cite[Proposition~6.4 and the proof of Theorem~6.5]{BartuselFuhrOussa2023}.
The unitary change of matrix coordinates \(S\) from that result splits the
conjugation representation into orthogonal invariant spaces
\[
 K_1\cong\C^{\Fps},
 \qquad
 K_2\cong\C^{\Fps\times(\Fp\setminus\{0,-1\})}.
\]
Write \(SA=A_1\oplus A_2\).  Applied to the rank-one seed, the same
coordinate change gives
\[
 S(\phi\phi^*)=(w_\phi\mid C'_\phi),
 \qquad
 w_\phi(m)=|\phi(-m)|^2,
\]
where \(C'_\phi\) is obtained from \(B_\phi\) by unitary row and column
permutations.  Hence
\[
        \sigma_{\min}(C'_\phi)=\sigma_{\min}(B_\phi)=\beta_\phi.
        \tag{1}
\]

Let \(F=\mathcal M_\phi(A)\), and denote by \(F_1,F_2\) the coefficient
functions arising from \(A_1,A_2\), respectively.  The explicit coordinate
action is
\[
 F_1(k,l)=\sum_{m\in\Fps} A_1(m)w_\phi(lm),
 \qquad
 F_2\perp\operatorname{span}\{\widetilde\chi:\chi\in\widehat{\Fps}\}.
        \tag{2}
\]
The first function is independent of \(k\), and therefore belongs to the
character span in (2).  Thus \(F_1\perp F_2\), and
\[
        \norm{F}_2^2=\norm{F_1}_2^2+\norm{F_2}_2^2.
        \tag{3}
\]

Use the unitary Fourier transform on the multiplicative group \(\Fps\).
The first formula in (2) is multiplicative convolution, with multiplier
\(c_\phi(\chi)\) at the character \(\chi\).  Since the additive parameter has
\(p\) values, Plancherel gives
\[
 \norm{F_1}_2^2
 =p\sum_{\chi\in\widehat{\Fps}}
      |c_\phi(\chi)|^2|\widehat{A_1}(\chi)|^2
 \geq p\alpha_\phi^2\norm{A_1}_F^2.
        \tag{4}
\]

For the second block, the same coordinate formula identifies \(F_2\) with
the coefficient transform of the irreducible \(n\)-dimensional representation
\(\pi\), with matrix coefficient \(A_2(C'_\phi)^*\).  Schur orthogonality
therefore gives
\[
 \norm{F_2}_2^2
 =\frac{|G|}{n}\norm{A_2(C'_\phi)^*}_F^2
 =p\norm{A_2(C'_\phi)^*}_F^2
 \geq p\beta_\phi^2\norm{A_2}_F^2,
        \tag{5}
\]
where the last inequality uses (1).  Combining (3)--(5), and using the
unitarity of \(S\), proves the asserted bound.
\end{proof}

\section{The Bartusel--F\"uhr--Oussa Generator}

We now analyze the explicit vector
\[
        \phi(1)=0,\qquad
        \phi(m)=\frac{1}{\sqrt{n-1}}\quad (m\neq 1).
\]
This is the qualitative example used in \cite{BartuselFuhrOussa2023}.  The
new point is that both blocks of the criterion have computable singular-value
lower bounds.

\begin{lemma}[The character block]
\label{lem:character}
For the above vector \(\phi\),
\[
        c_\phi(\mathbf 1)=1,
        \qquad
        |c_\phi(\chi)|=\frac{1}{n-1}
        \quad
        \text{for every nontrivial }\chi\in\widehat{\Fps}.
\]
Thus \(\alpha_\phi=1/(n-1)\).
\end{lemma}

\begin{proof}
The trivial character gives \(c_\phi(\mathbf 1)=\norm{\phi}_2^2=1\).  If
\(\chi\) is nontrivial, then
\[
        c_\phi(\chi)
        =
        \frac{1}{n-1}\sum_{\substack{m\in\Fps\\ -m\neq 1}}\chi(m)
        =
        -\frac{\chi(-1)}{n-1},
\]
because \(\sum_{m\in\Fps}\chi(m)=0\).  Taking absolute values proves the
claim.
\end{proof}

\begin{lemma}[The \(B\)-block is a path-complement matrix]
\label{lem:path}
Let \(E\) be the \(n\times(n-1)\) unoriented incidence matrix of the path graph
on \(n\) vertices.  Then \(B_\phi\) is obtained from
\[
        \frac{1}{n-1}(\mathbf 1-E)
\]
by row and column permutations.  Here \(\mathbf 1\) denotes the all-ones
\(n\times(n-1)\) matrix.  Consequently
\[
        \sigma_{\min}(B_\phi)
        \geq
        \frac{2}{n-1}\sin\left(\frac{\pi}{2n}\right).
\]
\end{lemma}

\begin{proof}
For \(t\in\Fp\setminus\{0,-1\}\), the column \(B_\phi(\cdot,t)\) vanishes
exactly at the two rows
\[
        m=t^{-1},\qquad m=(t+1)^{-1}.
\]
At all other rows it equals \(1/(n-1)\).  Thus the zero pattern is the
incidence pattern of the graph on \(\Fps\) whose edges are
\[
        \{t^{-1},(t+1)^{-1}\},
        \qquad t\in\Fp\setminus\{0,-1\}.
\]
It remains only to identify this graph.  For each
\(j\in\{1,\ldots,p-2\}\subset \Fp\), the column indexed by \(t=j\) gives the
edge
\[
        \left\{\frac1j,\frac1{j+1}\right\}.
\]
Thus, after the vertex relabelling
\[
        v_j=\frac1j,\qquad j=1,\ldots,p-1,
\]
the edge set is exactly
\[
        \{v_j,v_{j+1}\},\qquad j=1,\ldots,p-2.
\]
This is the path graph on \(n=p-1\) vertices.

Let \(C=\mathbf 1-E\).  Since every column of \(E\) has two ones,
\[
        C^*C
        =
        nJ-2J-2J+E^*E
        =
        (n-4)J+E^*E,
\]
where \(J\) is the all-ones matrix of size \(n-1\).  For a path incidence
matrix,
\[
        E^*E=2I+A_{P_{n-1}},
\]
where \(A_{P_{n-1}}\) is the adjacency matrix of the path on \(n-1\) vertices.
The rank-one term \((n-4)J\) is positive semidefinite for \(n\geq4\), so
\[
        \lambda_{\min}(C^*C)
        \geq
        \lambda_{\min}(2I+A_{P_{n-1}}).
\]
The eigenvalues of \(A_{P_{n-1}}\) are
\[
        2\cos\left(\frac{k\pi}{n}\right),
        \qquad k=1,\ldots,n-1.
\]
Hence
\[
        \lambda_{\min}(2I+A_{P_{n-1}})
        =
        2-2\cos\left(\frac{\pi}{n}\right)
        =
        4\sin^2\left(\frac{\pi}{2n}\right).
\]

Since \(B_\phi=C/(n-1)\), the stated singular-value lower bound follows.
\end{proof}

\begin{lemma}[Matching order for the \(B\)-block]
\label{lem:pathupper}
For the same vector \(\phi\),
\[
        \sigma_{\min}(B_\phi)
        \leq
        \frac{\sqrt3\,\pi}{(n-1)n}.
\]
\end{lemma}

\begin{proof}
Keep the notation \(B_\phi=C/(n-1)\) and
\[
        C^*C=(n-4)J+2I+A_{P_{n-1}}
\]
from Lemma \ref{lem:path}.  Let
\[
        z_j=(-1)^j\sin\left(\frac{j\pi}{n}\right),
        \qquad j=1,\ldots,n-1.
\]
Then \(z\) is an eigenvector of \(A_{P_{n-1}}\) with eigenvalue
\(-2\cos(\pi/n)\), and
\[
        \norm{z}_2^2=\frac n2.
\]
Thus
\[
        \frac{z^*(2I+A_{P_{n-1}})z}{\norm{z}_2^2}
        =
        4\sin^2\left(\frac{\pi}{2n}\right)
        \leq
        \frac{\pi^2}{n^2}.
\]
If \(n\) is odd, then \(\sum_j z_j=0\).  If \(n\) is even, the geometric-sum
identity gives
\[
        \left|\sum_{j=1}^{n-1}(-1)^j
        \sin\left(\frac{j\pi}{n}\right)\right|
        =
        \tan\left(\frac{\pi}{2n}\right).
\]
Since \(n\geq4\), \(\tan x\leq 2x\) for \(0\leq x\leq \pi/8\).  Hence in all
cases
\[
        \frac{(n-4)|\sum_j z_j|^2}{\norm{z}_2^2}
        \leq
        \frac{2\pi^2}{n^2}.
\]
Therefore
\[
        \frac{z^*C^*Cz}{\norm{z}_2^2}
        \leq
        \frac{3\pi^2}{n^2}.
\]
It follows that \(\sigma_{\min}(C)\leq \sqrt3\pi/n\), and the claim follows
from \(B_\phi=C/(n-1)\).
\end{proof}

\begin{proof}[Proof of Theorem \ref{thm:zero}]
Lemma \ref{lem:character} gives
\(\alpha_\phi=1/(n-1)\), while Lemma \ref{lem:path} gives
\[
        \beta_\phi\geq
        \frac{2}{n-1}\sin\left(\frac{\pi}{2n}\right).
\]
For \(n\geq4\), the latter is smaller.  Proposition \ref{prop:block} therefore
gives
\[
        \norm{\mathcal M_\phi(A)}_2
        \geq
        \sqrt p\,\frac{2}{n-1}
        \sin\left(\frac{\pi}{2n}\right)\norm{A}_F.
\]
For the upper estimate, restrict the measurement map to the \(K_2\)-block in
the proof of Proposition \ref{prop:block}.  On this block,
\[
        \norm{\mathcal M_\phi(A)}_2
        =
        \sqrt p\,\norm{A_2B_\phi^*}_F.
\]
Choose \(A_2\) to be a rank-one matrix aligned with a right singular vector of
\(B_\phi\) for \(\sigma_{\min}(B_\phi)\).  Lemma \ref{lem:pathupper} then gives
a nonzero matrix \(A\) satisfying
\[
        \norm{\mathcal M_\phi(A)}_2
        \leq
        \frac{\sqrt3\,\pi\sqrt p}{(n-1)n}\norm{A}_F.
\]
\end{proof}

\section{A Weighted Two-Level Generator}

The zero in the Bartusel--F\"uhr--Oussa vector is useful for proving
injectivity, but it is not a good choice for conditioning.  The same path
structure permits a direct deterministic repair.

\begin{lemma}[Two-level reduction]
\label{lem:twolevel}
Let \(r\geq0\), put \(b_r=(n-1+r^2)^{-1/2}\), and define
\[
        \phi_r(1)=r b_r,\qquad
        \phi_r(m)=b_r\quad(m\neq1).
\]
After row and column permutations,
\[
 B_{\phi_r}=b_r^2\bigl(\mathbf1+(r-1)E\bigr),
        \tag{6}
\]
where \(E\) is the unoriented incidence matrix of the path \(P_n\) used in
Lemma \ref{lem:path}.  Consequently, with \(s=r-1\),
\[
 \bigl(\mathbf1+sE\bigr)^*\bigl(\mathbf1+sE\bigr)
 =
 (n+4s)J+s^2(2I+A_{P_{n-1}}).
        \tag{7}
\]
\end{lemma}

\begin{proof}
For each column indexed by \(t\), the two rows \(t^{-1}\) and
\((t+1)^{-1}\) correspond to the two incidences of the associated path edge.
At these two entries \(B_{\phi_r}\) has value \(rb_r^2\), and at every other
entry it has value \(b_r^2\).  This proves (6).  Since every column of \(E\)
has two ones,
\[
 \mathbf1^*\mathbf1=nJ,\qquad
 \mathbf1^*E=E^*\mathbf1=2J,\qquad
 E^*E=2I+A_{P_{n-1}},
\]
which proves (7).
\end{proof}

\begin{theorem}[An explicit \(O(d)\)-stable affine orbit]
\label{thm:main}
Let \(p\geq5\) be prime, \(n=p-1\), and set
\[
 \phi_\star(1)=\frac{1+\sqrt n}{\sqrt{2n+2\sqrt n}},
 \qquad
 \phi_\star(m)=\frac{1}{\sqrt{2n+2\sqrt n}}
 \quad(m\neq1).
\]
Then, for every \(A\in\C^{\Fps\times\Fps}\),
\[
 \norm{\mathcal M_{\phi_\star}(A)}_{\ell^2(G)}
 \geq
 \frac{\sqrt p}{\sqrt n+1}
 \sin\left(\frac{\pi}{2n}\right)\norm{A}_F.
        \tag{8}
\]
In particular, the inverse on the range has norm at most \(2d\).
Conversely, there is a nonzero \(A\) such that
\[
 \norm{\mathcal M_{\phi_\star}(A)}_{\ell^2(G)}
 \leq
 \frac{\pi\sqrt{5p/2}}{2n(\sqrt n+1)}\norm{A}_F.
        \tag{9}
\]
Thus the least singular value of \(\mathcal M_{\phi_\star}\) is of order
\(d^{-1}\), and this order is sharp for \(\phi_\star\).
\end{theorem}

\begin{proof}
This is the choice \(r=1+\sqrt n\) in Lemma \ref{lem:twolevel}; hence
\[
        b_r^2=\frac{1}{2n+2\sqrt n},
        \qquad
        B_{\phi_\star}=b_r^2(\mathbf1+\sqrt n E).
        \tag{10}
\]
For a nontrivial multiplicative character,
\[
 c_{\phi_\star}(\chi)
 =
 \bigl(\phi_\star(1)^2-\phi_\star(m)^2\bigr)\chi(-1),
\]
and therefore
\[
 \alpha_{\phi_\star}
 =
 \frac{n+2\sqrt n}{2n+2\sqrt n}
 =
 \frac{\sqrt n+2}{2(\sqrt n+1)}.
        \tag{11}
\]
By (7), the Gram matrix in (10) is bounded below by
\[
        n(2I+A_{P_{n-1}}),
\]
because its \(J\)-coefficient \(n+4\sqrt n\) is nonnegative.  The path
spectrum used in Lemma \ref{lem:path} now gives
\[
 \beta_{\phi_\star}
 \geq
 2b_r^2\sqrt n\sin\left(\frac{\pi}{2n}\right)
 =
 \frac{1}{\sqrt n+1}\sin\left(\frac{\pi}{2n}\right).
        \tag{12}
\]
The quantity in (12) is smaller than (11).  Proposition \ref{prop:block}
therefore proves (8).  Since
\(\sin(\pi/(2n))\geq1/n\) and
\(\sqrt{n+1}/(\sqrt n+1)\geq1/2\), the lower bound in (8) is at least
\((2n)^{-1}\), proving the inverse estimate.

It remains to prove the matching upper estimate.  Let
\[
        z_j=(-1)^j\sin\left(\frac{j\pi}{n}\right),
        \qquad j=1,\ldots,n-1.
\]
As in Lemma \ref{lem:pathupper},
\[
 \norm z_2^2=\frac n2,\qquad
 \frac{z^*(2I+A_{P_{n-1}})z}{\norm z_2^2}
 \leq\frac{\pi^2}{n^2}.
        \tag{13}
\]
Moreover, \(\sum_jz_j=0\) for odd \(n\), while for even \(n\) its absolute
value is \(\tan(\pi/(2n))\leq\pi/n\).  Hence
\[
 \frac{(n+4\sqrt n)|\sum_jz_j|^2}{\norm z_2^2}
 \leq\frac{6\pi^2}{n^2}.
        \tag{14}
\]
Equations (7), (13), and (14) yield
\[
 \frac{z^*(\mathbf1+\sqrt nE)^*(\mathbf1+\sqrt nE)z}{\norm z_2^2}
 \leq
 \frac{5\pi^2}{2n}.
\]
Together with (10), this implies
\[
 \beta_{\phi_\star}
 \leq
 \frac{\pi\sqrt{5/2}}{2n(\sqrt n+1)}.
\]
Finally restrict the measurement operator to the \(K_2\)-block and choose a
rank-one matrix aligned with a right singular vector for this least singular
value.  The exact block identity (5) gives (9).
\end{proof}

\begin{theorem}[Optimal stability order in the two-level family]
\label{thm:twoleveloptimal}
For \(r\geq0\), let \(\phi_r\) be as in Lemma \ref{lem:twolevel}, and let
\(\gamma_r\) denote the least singular value of \(\mathcal M_{\phi_r}\).  Then
\[
 \gamma_r\geq
 2\sqrt p\,\frac{|r-1|}{n-1+r^2}
 \sin\left(\frac{\pi}{2n}\right).
        \tag{15}
\]
For every \(r\geq0\), one also has
\[
        \gamma_r\leq\frac{3\pi}{n}.
        \tag{16}
\]
Moreover,
\[
 \max_{r\geq0}\frac{|r-1|}{n-1+r^2}
 =
 \max\left\{\frac1{n-1},\frac1{2(\sqrt n+1)}\right\}.
        \tag{17}
\]
The first value in (17) is attained at \(r=0\), and the second at
\(r=1+\sqrt n\).  In particular, for \(n\geq10\), the vector
\(\phi_\star\) has the optimal stability order \(d^{-1}\) within the entire
two-level family.
\end{theorem}

\begin{proof}
Put \(D_r=n-1+r^2\) and \(s=r-1\).  For nontrivial characters,
\[
        |c_{\phi_r}(\chi)|
        =
        \frac{|r^2-1|}{D_r}.
\]
The right side is at least
\(2|r-1|\sin(\pi/(2n))/D_r\), because \(r+1\geq1\) and
\(2\sin(\pi/(2n))<1\).  The same lower bound for the \(B\)-block follows
from (7), since \(n+4s\geq n-4\geq0\):
\[
 \sigma_{\min}(B_{\phi_r})
 \geq
 2b_r^2|s|\sin\left(\frac{\pi}{2n}\right)
 =
 2\frac{|r-1|}{D_r}\sin\left(\frac{\pi}{2n}\right).
\]
Proposition \ref{prop:block} proves (15).

For the upper estimate, use the vector \(z\) from the proof of Theorem
\ref{thm:main}.  Since \(p\) is odd, \(n\) is even, and the geometric-sum
identity used there gives \(|\sum_jz_j|\leq\pi/n\).  Equation (7) yields
\[
 \sigma_{\min}\bigl(\mathbf1+sE\bigr)
 \leq
 \frac{\pi|s|}{n}
 +
 \frac{\pi\sqrt{2(n+4s)}}{n^{3/2}}.
        \tag{18}
\]
The restriction to the \(K_2\)-block and the exact identity (5) therefore
show that
\[
 \gamma_r
 \leq
 \sqrt p\,b_r^2
 \left(
 \frac{\pi|r-1|}{n}
 +
 \frac{\pi\sqrt{2(n+4r-4)}}{n^{3/2}}
 \right).
        \tag{19}
\]
The elementary estimates
\[
 \frac{|r-1|}{D_r}\leq\frac1{\sqrt{n-1}},
 \qquad
 \frac{\sqrt{n+4r-4}}{D_r}\leq\frac{\sqrt2}{\sqrt{n-1}}
\]
hold for all \(r\geq0\) and \(n\geq4\).  Using
\(\sqrt{n+1}/\sqrt{n-1}<3/2\) in (19) gives (16).

Finally, \(f(r)=|r-1|/D_r\) decreases on \([0,1]\).  On \([1,\infty)\),
differentiation shows that its unique maximum occurs at \(r=1+\sqrt n\).
The two resulting values are exactly those in (17).  For \(n\geq10\), the
second is the larger one.  Combining (15), (16), and the estimate
\(\sin(\pi/(2n))\geq1/n\) gives the last assertion.
\end{proof}

\section{Prime-Power Affine Orbits}

The preceding proof is not confined to prime fields.  The representation-theoretic
block calculation is algebraic over a field, while the path graph is replaced by
a graph with an equally explicit finite-geometric decomposition.

Let \(q=p^h\), where \(p\) is an odd prime, and put \(d=q-1\).  Thus \(d\)
is the prime-power analogue of the prime-field notation \(n=p-1\) used
above.  Write
\[
        G_q=\mathbb F_q\rtimes\mathbb F_q^\times,
        \qquad
        (\pi_q(k,l)f)(m)=\psi(km)f(lm),
\]
where \(\psi\) is any nontrivial additive character of \(\mathbb F_q\).
For completeness, \(\pi_q\) is irreducible: the translation subgroup acts
diagonally with the distinct characters \(m\mapsto\psi(km)\); its Fourier
idempotents therefore project any invariant subspace onto the coordinate axes.
Since \(\mathbb F_q^\times\) acts transitively on those axes, a nonzero
invariant subspace contains them all.

\begin{proposition}[Affine block criterion over \(\mathbb F_q\)]
\label{prop:qblock}
For a real unit vector \(\phi\in\mathbb R^{\mathbb F_q^\times}\), define
\[
 c_\phi(\chi)=\sum_{m\in\mathbb F_q^\times}|\phi(-m)|^2\chi(m),
 \qquad
 B_\phi(m,t)=\phi(mt)\phi(m(t+1)).
\]
Here \(\chi\) ranges over the multiplicative characters and
\(m\in\mathbb F_q^\times\), \(t\in\mathbb F_q\setminus\{0,-1\}\).
Then the lower bound is exact:
\[
 \norm{\mathcal M_\phi(A)}_{\ell^2(G_q)}
 \geq
 \sqrt q\,
 \min\left\{\min_\chi|c_\phi(\chi)|,\sigma_{\min}(B_\phi)\right\}
 \norm A_F.
        \tag{20}
\]
Equivalently,
\[
 \sigma_{\min}(\mathcal M_\phi)
 =\sqrt q\min\left\{\min_\chi|c_\phi(\chi)|,
 \sigma_{\min}(B_\phi)\right\}.
 \tag{20a}
\]
\end{proposition}

\begin{proof}
We give the reduction, since this is the point at which the extension from a
prime field to an arbitrary finite field has to be checked.  Put
\(\rho_q(k,l)A=\pi_q(k,l)A\pi_q(k,l)^*\).  On the matrix space, the
coordinate permutation
\[
 (SA)(m,n)=
 \begin{cases}
 A(-m,-m),&n=1,\\
 A\bigl(m(1-n)^{-1},mn(1-n)^{-1}\bigr),&n\neq1,
 \end{cases}
        \tag{21}
\]
is unitary.  Indeed, for \(n\ne1\), the two coordinates on the right hand
side have difference \(m\) and quotient \(n\), so (21) is a bijective
relabeling of \((\mathbb F_q^\times)^2\).  Write
\(\widetilde\rho_q=S\rho_qS^{-1}\).  Direct substitution gives
\[
 (\widetilde\rho_q(k,l)A)(m,n)=
 \begin{cases}
 A(lm,1),&n=1,\\
 \psi(km)A(lm,n),&n\neq1.
 \end{cases}
        \tag{22}
\]
Thus the first coordinate column is the regular representation of
\(\mathbb F_q^\times\), while the remaining \(d-1\) columns are copies of
the irreducible representation \(\pi_q\).

Apply (21) to \(\phi\phi^*\).  Its first column is
\((|\phi(-m)|^2)_m\).  For \(n\in\mathbb F_q^\times\setminus\{1\}\), put
\(t=n(1-n)^{-1}\).  This is a bijection from
\(\mathbb F_q^\times\setminus\{1\}\) onto
\(\mathbb F_q\setminus\{0,-1\}\), and
\[
 \phi\bigl(m(1-n)^{-1}\bigr)
 \phi\bigl(mn(1-n)^{-1}\bigr)
 =\phi\bigl(m(t+1)\bigr)\phi(mt)=B_\phi(m,t).
 \tag{23}
\]
Hence no unproved identification of the second block is being used: after
this explicit column relabeling it is exactly \(B_\phi\).

Let \(A_1\) and \(A_2\) denote the two pieces of \(SA\), and write
\(\mathcal M_\phi(A)=F_1+F_2\) for the corresponding decomposition of the
coefficient function.  The first block is multiplicative convolution.  Finite Fourier inversion on
\(\mathbb F_q^\times\) gives the exact identity
\[
 \|F_1\|_2^2
 =q\sum_\chi |c_\phi(\chi)|^2|\widehat A_1(\chi)|^2.
 \tag{24}
\]
For the second block, Schur orthogonality for the \(d\)-dimensional
irreducible \(\pi_q\) gives
\[
 \|F_2\|_2^2=q\|A_2B_\phi^*\|_F^2.
 \tag{25}
\]
Finally, \(F_1\) is independent of the additive coordinate \(k\), whereas
every term of \(F_2\) has a nontrivial additive character in that coordinate;
finite additive Fourier orthogonality gives \(F_1\perp F_2\).  Combining
(24)--(25), using that \(S\) is unitary, proves (20).  More precisely, the
two orthogonal blocks have singular values \(\sqrt q|c_\phi(\chi)|\) and
\(\sqrt q\) times the singular values of \(B_\phi\), respectively.  Their
union is the singular spectrum of \(\mathcal M_\phi\), which proves (20a).
\end{proof}

\begin{lemma}[The finite-field path--cycle graph]
\label{lem:pathcycle}
Let \(E_q\) be the unoriented incidence matrix whose columns are indexed by
\(t\in\mathbb F_q\setminus\{0,-1\}\) and whose \(t\)-th column has ones at
the vertices \(t^{-1}\) and \((t+1)^{-1}\).  Then \(E_q\) is, up to row and
column permutations, the incidence matrix of
\[
        P_{p-1}\ \sqcup\ \left(\frac{q}{p}-1\right)C_p.
        \tag{26}
\]
Consequently,
\[
        \sigma_{\min}(E_q)
        \geq
        2\sin\left(\frac{\pi}{2p}\right).
        \tag{27}
\]
If \(h\geq2\), then equality holds in (27).
\end{lemma}

\begin{proof}
Relabel the vertices by \(v_t=t^{-1}\).  The columns then become the edges
\(\{t,t+1\}\) on \(\mathbb F_q^\times\).  Translation by \(1\) partitions
\(\mathbb F_q\) into its additive \(\mathbb F_p\)-cosets.  The zero coset has
the vertex \(0\) removed and yields the path \(P_{p-1}\).  Every other coset
yields a \(p\)-cycle.  Since \(p\) is odd, no two values of \(t\) produce the
same unoriented edge, proving (26).

For a path \(P_{p-1}\), the least eigenvalue of \(E^*E\) is
\(4\sin^2(\pi/(2(p-1)))\).  For an odd cycle \(C_p\), the least eigenvalue is
\(4\sin^2(\pi/(2p))\).  The latter is smaller, and (27) follows.
When \(h\geq2\), at least one cycle is present, which also proves equality.
\end{proof}

\begin{theorem}[Uniformly stable affine orbits in prime-power dimensions]
\label{thm:primepower}
Let \(q=p^h\) with \(p\) odd, \(d=q-1\), and define
\[
 \phi_q(1)=\frac{1+\sqrt d}{\sqrt{2d+2\sqrt d}},
 \qquad
 \phi_q(m)=\frac{1}{\sqrt{2d+2\sqrt d}}\quad(m\neq1).
\]
Then
\[
 \norm{\mathcal M_{\phi_q}(A)}_{\ell^2(G_q)}
 \geq
 \frac{\sqrt q}{\sqrt d+1}
 \sin\left(\frac{\pi}{2p}\right)\norm A_F
 \qquad(A\in\C^{d\times d}).
        \tag{28}
\]
In particular, the inverse on the range has norm at most \(2p\).  Hence, for
each fixed odd characteristic \(p\), this produces a uniformly stable
deterministic matrix-recovery orbit in every dimension \(d=p^h-1\).
If \(h\geq2\), then the least singular value is exactly
\[
 \sigma_{\min}(\mathcal M_{\phi_q})
 =
 \frac{\sqrt q}{\sqrt d+1}
 \sin\left(\frac{\pi}{2p}\right).
        \tag{29}
\]
\end{theorem}

\begin{proof}
Put \(b=(2d+2\sqrt d)^{-1/2}\).  Lemma \ref{lem:pathcycle} and the
two-level identity give
\[
        B_{\phi_q}=b^2(\mathbf1+\sqrt d\,E_q).
\]
Since
\[
 (\mathbf1+\sqrt d\,E_q)^*(\mathbf1+\sqrt d\,E_q)
 =
 (d+4\sqrt d)J+dE_q^*E_q,
\]
the coefficient of \(J\) is exactly as displayed: \(\mathbf1\) has \(d\)
rows, so \(\mathbf1^*\mathbf1=dJ\), while every column of \(E_q\) has two
ones and hence \(\mathbf1^*E_q=2J\).
Lemma \ref{lem:pathcycle} implies
\[
 \sigma_{\min}(B_{\phi_q})
 \geq
 2b^2\sqrt d\sin\left(\frac{\pi}{2p}\right)
 =
 \frac{1}{\sqrt d+1}\sin\left(\frac{\pi}{2p}\right).
\]
The nontrivial character coefficients have modulus
\[
        \frac{d+2\sqrt d}{2d+2\sqrt d},
\]
which is larger than the preceding lower bound.  Proposition
\ref{prop:qblock} proves (28).  Finally,
\(\sin(\pi/(2p))\geq1/p\) and
\(\sqrt q/(\sqrt d+1)\geq1/2\), yielding the stated inverse bound.

Assume finally that \(h\geq2\).  Choose a cyclic Fourier eigenvector
on one \(C_p\)-component which realizes the least eigenvalue of its incidence
Gram matrix.  It has coordinate sum zero.  Extending it by zero to all other
edge coordinates annihilates the \(J\)-term, and hence realizes equality in
the preceding lower bound for \(\sigma_{\min}(B_{\phi_q})\).  The nontrivial
character block is strictly larger.  If \(z\) is this right singular vector
and \(u\ne0\) is arbitrary, take \(A_1=0\) and \(A_2=uz^*\) in the
coordinates of Proposition \ref{prop:qblock}.  Identity (25) then gives
\[
 \frac{\|\mathcal M_{\phi_q}(A)\|_2}{\|A\|_F}
 =\sqrt q\,\sigma_{\min}(B_{\phi_q}).
\]
Together with (28), this gives the reverse inequality and proves (29).
\end{proof}

\begin{theorem}[Exact optimization in the prime-power two-level family]
\label{thm:ppoptimal}
Assume \(h\geq2\).  For \(r\geq0\), put \(D_r=d-1+r^2\) and define
\[
 \phi_{q,r}(1)=\frac r{\sqrt{D_r}},
 \qquad
 \phi_{q,r}(m)=\frac1{\sqrt{D_r}}\quad(m\ne1).
\]
Then
\[
 \sigma_{\min}(\mathcal M_{\phi_{q,r}})
 =2\sqrt q\,\frac{|r-1|}{D_r}
 \sin\left(\frac{\pi}{2p}\right).
 \tag{30}
\]
Moreover,
\[
 \max_{r\geq0}\sigma_{\min}(\mathcal M_{\phi_{q,r}})
 =\frac{\sqrt q}{\sqrt d+1}\sin\left(\frac{\pi}{2p}\right)
 \tag{31}
\]
whenever \(d\geq10\), and the unique maximizing parameter is
\(r=1+\sqrt d\).  This uniqueness is for the displayed two-level family
with distinguished coordinate \(1\).  Choosing another distinguished
coordinate merely applies a multiplicative coordinate permutation, and gives
a unitarily equivalent orbit measurement problem.
\end{theorem}

\begin{proof}
Write \(s=r-1\).  The nontrivial character coefficients have modulus
\[
 \frac{|r^2-1|}{D_r}=\frac{|s|(r+1)}{D_r}
 \geq 2\frac{|s|}{D_r}\sin\left(\frac{\pi}{2p}\right).
 \tag{32}
\]
Here we used \(r+1\geq1\) and \(2\sin(\pi/(2p))\leq1\).  On the other
hand,
\[
 B_{\phi_{q,r}}=D_r^{-1}(\mathbf1+sE_q),
 \qquad
 (\mathbf1+sE_q)^*(\mathbf1+sE_q)
 =(d+4s)J+s^2E_q^*E_q.
 \tag{33}
\]
Since \(d+4s\geq d-4>0\), Lemma \ref{lem:pathcycle} gives the lower
bound
\[
 \sigma_{\min}(B_{\phi_{q,r}})
 \geq2\frac{|s|}{D_r}\sin\left(\frac{\pi}{2p}\right).
 \tag{34}
\]
The cycle eigenvector used in the proof of Theorem \ref{thm:primepower} has
coordinate sum zero, so it annihilates the \(J\)-term in (33).  Thus equality
holds in (34).  The trivial character coefficient is one, while (32) controls
every nontrivial character coefficient, so the character block is no smaller.
The exact block identities from Proposition \ref{prop:qblock}, applied to a
rank-one matrix supported on the second block, prove (30).

It remains only to maximize \(|r-1|/D_r\).  It decreases on \([0,1]\),
with maximum \(1/(d-1)\) at \(r=0\); on \([1,\infty)\) it has its unique
maximum \(1/(2(\sqrt d+1))\) at \(r=1+\sqrt d\).  The latter is strictly
larger when \(d\geq10\), proving (31).
\end{proof}

\begin{remark}[What remains open over a prime field]
\label{rem:primefieldlimitation}
The exact formula in Theorem~\ref{thm:ppoptimal} uses the cycle component in
Lemma~\ref{lem:pathcycle}, and is therefore asserted only for \(h\geq2\).
When \(h=1\), the graph is the path \(P_{p-1}\); the extremal path vector
does not in general annihilate the \(J\)-term in \((33)\).  The earlier
prime-field theorems give matching stability order in that case, but this
paper does not claim a sharp two-level constant or an exact optimizer for the
prime-field path problem.
\end{remark}

\begin{theorem}[Trace-flat affine orbits on every \(3\)-power tower]
\label{thm:traceflat}
Let \(q=3^h\) with \(h\geq2\), let
\(\tau=\operatorname{Tr}_{\mathbb F_q/\mathbb F_3}\), and put
\[
 \eta=2-\frac{3\sqrt2}{2},
 \qquad
 r=\frac{-1+\sqrt{9-6\sqrt2}}2,
 \qquad
 a=\frac q3,
 \qquad
 N_q=2a-1+2a\eta=q(2-\sqrt2)-1.
\]
The real unit vector
\[
 \phi_q(m)=\frac1{\sqrt{N_q}}
 \begin{cases}
  1,&\tau(m)=0,\\
  r,&\tau(m)=1,\\
  -1-r,&\tau(m)=2
 \end{cases}
 \qquad(m\in\mathbb F_q^\times)
\]
has matching character and second-block minima:
\[
 \min_{\chi\ne\mathbf1}|c_{\phi_q}(\chi)|
 =\sigma_{\min}(B_{\phi_q})
 =\frac{\sqrt{a(1+4\eta)}}{N_q}.
\]
Consequently,
\[
 \sigma_{\min}(\mathcal M_{\phi_q})
 =\frac{q(\sqrt2-1)}{q(2-\sqrt2)-1}
 >\frac1{\sqrt2}.
\]
In particular, these exact stability constants decrease to \(1/\sqrt2\) as
\(h\to\infty\).
\end{theorem}

\begin{proof}
Write \(x=1\), \(y=r\), and \(z=-1-r\), so that
\(x+y+z=0\) and \(r(r+1)=\eta\).  The nonzero elements of each nonzero trace
class have cardinality \(a\), while the trace-zero class has cardinality
\(a-1\).  This gives
\[
 \|\phi_q\|^2=\frac{(a-1)x^2+ay^2+az^2}{N_q}=1.
\]

For \(t\notin\mathbb F_3\), the two \(\mathbb F_3\)-linear functionals
\(m\mapsto\tau(m)\) and \(m\mapsto\tau(tm)\) are independent.  Indeed, a
linear dependence would say \(\tau((u+vt)m)=0\) for every \(m\) and some
\((u,v)\ne(0,0)\) in \(\mathbb F_3^2\).  Nondegeneracy of the trace pairing
would give \(u+vt=0\), contrary to \(t\notin\mathbb F_3\).  Thus each pair
of trace values occurs \(q/9=a/3\) times before removal of \(m=0\).

Let \(R(t)=N_q^2\sum_m\phi_q(m)^2\phi_q(tm)^2\).  The preceding count,
together with the immediate cases \(t=1\) and \(t=-1\), gives
\[
 R(1)=2a-1+4a\eta+2a\eta^2,
 \qquad
 R(-1)=a-1+2a\eta^2,
\]
and, for \(t\notin\mathbb F_3\),
\[
 R(t)=\frac{4a}{3}-1+\frac{4a}{3}\eta^2+\frac{8a}{3}\eta.
\]
Since \(2\eta^2-8\eta-1=0\), the latter two values are equal; call their common
value \(R_0\).  Moreover \(R(1)-R_0=a(1+4\eta)\).  Multiplicative Fourier
inversion now yields
\[
 |c_{\phi_q}(\chi)|^2=\frac{R(1)-R_0}{N_q^2}
 =\frac{a(1+4\eta)}{N_q^2}
\]
for every nontrivial multiplicative character.  Replacing \(m\) by \(-m\)
accounts for the convention in \(c_{\phi_q}\) and leaves the calculation
unchanged.

It remains to calculate the second block.  Let
\(\zeta=e^{2\pi i/3}\) and \(\psi(u)=\zeta^{\tau(u)}\).  Since
\(x+y+z=0\), there is a complex number
\[
 \gamma=\frac12-\frac{i(1+2r)}{2\sqrt3}
\]
such that the unnormalized three-level function is
\(F(u)=\gamma\psi(u)+\overline\gamma\psi(-u)\).  Put
\(c=\gamma^2\) and \(e=|\gamma|^2\).  A direct calculation gives
\[
 \operatorname{Re}(\gamma^4)
 =\frac{2\eta^2-8\eta-1}{18}=0,
 \qquad
 e=\frac{1+\eta}{3}.
\]
Reindex the columns of the unnormalized second block by
\(\lambda=2t+1\).  As \(t\) ranges over
\(\mathbb F_q\setminus\{0,-1\}\), \(\lambda\) ranges over
\(\mathbb F_q\setminus\{1,-1\}\), and the corresponding column is
\[
 g_\lambda(m)=c\psi(\lambda m)+e\psi(-m)+e\psi(m)
                 +\overline c\psi(-\lambda m)
 \qquad(m\in\mathbb F_q^\times).
\]
For \(u\in\mathbb F_q\), additive-character orthogonality says
\[
 \sum_{m\in\mathbb F_q^\times}\psi(um)
 =\begin{cases}q-1,&u=0,\\-1,&u\ne0.\end{cases}
\]
The sum of the four coefficients in \(g_\lambda\) is
\(c+\overline c+2e=(\gamma+\overline\gamma)^2=1\).  Thus the diagonal
entries of the unnormalized Gram matrix are \(2q(e^2+|c|^2)-1\) for
\(\lambda\ne0\), whereas the \(\lambda=0\) diagonal entry is
\(q((c+\overline c)^2+2e^2)-1\).  The latter differs from the former by
\(2q\operatorname{Re}(c^2)\).  For two distinct nonzero indices, the
off-diagonal entry is \(2qe^2-1\) when \(\lambda\ne-\mu\), and
\(2q(e^2+\operatorname{Re}(c^2))-1\) when \(\lambda=-\mu\).  The
off-diagonal entry between \(0\) and a nonzero index is \(2qe^2-1\).
The vanishing of \(\operatorname{Re}(c^2)=\operatorname{Re}(\gamma^4)\)
therefore makes all diagonal entries equal and all off-diagonal entries equal.
Consequently,
\[
 B_{\phi_q}^*B_{\phi_q}
 =\frac{2q|\gamma|^4}{N_q^2}I_{q-2}
  +\frac{2qe^2-1}{N_q^2}J_{q-2}.
\]
The coefficient of \(J_{q-2}\) is positive for \(q\geq9\), so the least
eigenvalue is \(2q|\gamma|^4/N_q^2\).  Finally, the equation defining
\(\eta\) gives
\[
 2q|\gamma|^4=\frac{2q(1+\eta)^2}{9}
 =\frac q3(1+4\eta)=a(1+4\eta).
\]
This proves the asserted equality of the two block minima.  Their common
value is less than one, while the trivial character coefficient is one, so
the exact block criterion proves the displayed formula for
\(\sigma_{\min}(\mathcal M_{\phi_q})\).  Finally,
\(1+4\eta=3(\sqrt2-1)^2\), and hence
\(\sqrt{(1+4\eta)/3}=\sqrt2-1\), while
\(N_q=q(2-\sqrt2)-1\).  The displayed value is larger than
\(1/\sqrt2\), and its derivative as a function of real \(q>0\) is negative,
so it decreases to \(1/\sqrt2\).
\end{proof}

\begin{lemma}[Two-valued character spectrum for real trace windows]
\label{lem:tracecharacters}
Let \(q=3^h\) with \(h\geq2\), put \(a=q/3\), and let \(F\not\equiv0\) take the
three real values \(F_0,F_1,F_2\) on the three trace classes, respectively.
Let \(\phi=F/\sqrt N\) on \(\mathbb F_q^\times\), where
\[
 N=(a-1)F_0^2+aF_1^2+aF_2^2.
\]
Define
\[
 \begin{aligned}
 A&=a(F_0^4+F_1^4+F_2^4)-F_0^4,\\
 B&=(a-1)F_0^4+2aF_1^2F_2^2,\\
 C&=\frac a3(F_0^2+F_1^2+F_2^2)^2-F_0^4.
 \end{aligned}
 \tag{35}
\]
Then every nontrivial multiplicative character \(\chi\) satisfies
\[
 |c_\phi(\chi)|^2
 =\frac{A-C+\chi(-1)(B-C)}{N^2}.
 \tag{36}
\]
Thus the nontrivial character spectrum has at most two values, distinguished
only by the parity \(\chi(-1)\).
\end{lemma}

\begin{proof}
For \(t\in\mathbb F_q^\times\), put
\(R(t)=\sum_{m\in\mathbb F_q^\times}F(m)^2F(tm)^2\).  Directly from the
sizes of the trace classes,
\[
 R(1)=A,\qquad R(-1)=B.
\]
If \(t\notin\mathbb F_3\), the two trace functionals
\(m\mapsto\operatorname{Tr}(m)\) and
\(m\mapsto\operatorname{Tr}(tm)\) are independent.  Every ordered pair of
trace values therefore occurs \(q/9=a/3\) times before removing \(m=0\),
which gives \(R(t)=C\).  Multiplicative Fourier inversion gives
\[
 N^2|c_\phi(\chi)|^2=\sum_{t\in\mathbb F_q^\times}R(t)\chi(t).
\]
For nontrivial \(\chi\), the sum of \(\chi(t)\) over
\(t\notin\{1,-1\}\) is \(-1-\chi(-1)\).  Substitution proves (36).
\end{proof}

\begin{lemma}[Visible spectral branches of a zero-sum trace window]
\label{lem:tracebranches}
Let \(q=3^h\) with \(h\geq2\), let
\(\psi(u)=\exp(2\pi i\operatorname{Tr}_{\mathbb F_q/\mathbb F_3}(u)/3)\),
and let
\[
 F(u)=\gamma\psi(u)+\overline\gamma\psi(-u),
 \qquad \gamma=\rho e^{i\theta}\ne0.
\]
Normalize the restriction of \(F\) to \(\mathbb F_q^\times\) by its squared
norm \(N=2\rho^2(q-2\cos^2\theta)\).  In the second block associated with
\(F/\sqrt N\), indexed by
\(L=\mathbb F_q\setminus\{1,-1\}\), the following are eigenvalues of the
Gram matrix:
\[
 \begin{aligned}
 \frac{4q\rho^4\sin^2(2\theta)}{N^2}
 &\quad\text{on a subspace of dimension }\frac{q-3}{2},\\
 \frac{4q\rho^4\cos^2(2\theta)}{N^2}
 &\quad\text{on a subspace of dimension }\frac{q-5}{2}.
 \end{aligned}
 \tag{37}
\]
The first invariant subspace consists of vectors that are antisymmetric under
\(\lambda\mapsto-\lambda\) and vanish at \(0\).  The second invariant subspace consists of
vectors that are symmetric under this involution, vanish at \(0\), and have
coordinate sum zero.
\end{lemma}

\begin{proof}
Put \(C=\gamma^2\) and \(E=|\gamma|^2=\rho^2\).  After the change of
variables \(\lambda=2t+1\), the unnormalized column indexed by
\(\lambda\in L\) is
\[
 g_\lambda(m)=C\psi(\lambda m)+E\psi(-m)+E\psi(m)
                    +\overline C\psi(-\lambda m).
\]
Additive-character orthogonality on \(\mathbb F_q^\times\), together with
\(C+\overline C+2E=F(0)^2\), gives the following four types of Gram entries:
Put \(Q=F(0)^4=16\rho^4\cos^4\theta\).  Then
\[
 \begin{array}{c|c}
 \text{entry}&\text{unnormalized value}\\ \hline
 \lambda=\mu\ne0& D=4q\rho^4-Q\\
 \lambda=\mu=0&D_0=D+2q\rho^4\cos(4\theta)\\
 \lambda,\mu\ne0,\ \lambda\ne\pm\mu&O=2q\rho^4-Q\\
 \lambda=-\mu\ne0&O'=O+2q\rho^4\cos(4\theta).
 \end{array}
\]
The entries joining \(0\) to a nonzero index are also \(O\).  On an
antisymmetric vector, the \(0\)-coordinate is zero and the eigenvalue is
\(D-O'\).  On a symmetric vector with zero \(0\)-coordinate and zero
coordinate sum, it is \(D+O'-2O\).  The identities
\[
 D-O'=4q\rho^4\sin^2(2\theta),
 \qquad
 D+O'-2O=4q\rho^4\cos^2(2\theta)
\]
give (37) after division by \(N^2\).  Counting the nonzero elements of
\(L\) in opposite pairs gives the stated invariant-subspace dimensions; eigenvalues may coincide.
\end{proof}

\begin{theorem}[Complete spectral reduction in the zero-sum trace class]
\label{thm:tracespectrum}
Under the hypotheses and notation of Lemmas~\ref{lem:tracecharacters} and
\ref{lem:tracebranches}, put \(\kappa=\cos(4\theta)\) and \(Q=16\rho^4\cos^4\theta\), and let \(\mu_-\)
denote the smaller eigenvalue of the explicit real symmetric matrix
\[
 K_{q,\rho,\theta}=\frac1{N^2}
 \begin{pmatrix}
  2q\rho^4(2+\kappa)-Q&\sqrt{q-3}\,(2q\rho^4-Q)\\
  \sqrt{q-3}\,(2q\rho^4-Q)&
  2q\rho^4(q-2+\kappa)-(q-3)Q
 \end{pmatrix}.
 \tag{38}
\]
Then the least singular value of the affine matrix-recovery map is given
exactly by
\[
 \sigma_{\min}(\mathcal M_\phi)^2
 =q\min\left\{
 \frac{A-C-|B-C|}{N^2},\
 \frac{4q\rho^4}{N^2}\min\{\sin^2(2\theta),\cos^2(2\theta)\},\
 \mu_-
 \right\}.
 \tag{39}
\]
In particular, the conditioning of every real zero-sum trace window is
decided by two character levels, two visible spectral branches, and one
explicit \(2\times2\) matrix.
\end{theorem}

\begin{proof}
Lemma~\ref{lem:tracecharacters} gives the first term in (39).  The two
eigenvalues in Lemma~\ref{lem:tracebranches} account for the antisymmetric
subspace and for the symmetric, zero-sum subspace.  Their orthogonal
complement is two-dimensional: it is spanned by the coordinate vector at
\(\lambda=0\) and the normalized all-ones vector on the
\((q-3)/2\) opposite pairs of nonzero indices.

With the notation \(D,D_0,O,O'\) used in the proof of
Lemma~\ref{lem:tracebranches}, the Gram matrix on this complement is
\[
 \frac1{N^2}
 \begin{pmatrix}
 D_0&\sqrt{q-3}\,O\\
 \sqrt{q-3}\,O&D+O'+(q-5)O
 \end{pmatrix}.
\]
Substituting the four displayed values there gives exactly (38).  Hence its
smaller eigenvalue is \(\mu_-\).  The three invariant pieces exhaust the
\(q-2\) columns of the second block.  The exact block criterion in
Proposition~\ref{prop:qblock}, together with the fact that
\(c_\phi(\mathbf1)=1\), now yields (39).  Both parities occur among
nontrivial multiplicative characters for \(q\geq9\).  Their moduli are at
most one by the triangle inequality, so omitting the trivial character
does not change the minimum.
\end{proof}

\begin{theorem}[Spectral reduction for arbitrary real trace windows]
\label{thm:generaltrace}
Let \(q=3^h\), \(h\geq2\), and let
\[
 F(u)=a+\gamma\psi(u)+\overline\gamma\psi(-u),
 \qquad a\in\mathbb R,\quad \gamma=x+iy,
\]
where \(\psi(u)=\exp(2\pi i\operatorname{Tr}(u)/3)\).
Assume \(F\not\equiv0\), and normalize its restriction to
\(\mathbb F_q^\times\) as \(\phi=F/\sqrt N\).  Define
\[
 \begin{gathered}
 e=x^2+y^2,\quad b=2(x^2-y^2),\quad c=2ax,\quad n=q-3,\\
 N=q(a^2+2e)-(a+2x)^2,\qquad Q=(a+2x)^4,\\
 V=(a^2+b)^2+2(e+c)^2,\quad
 W=a^4+2e^2,\quad Z=a^2(a^2+b)+2e(e+c),\\
 L_1=b+2c,\quad L_2=4y(x-a),\quad
 L_3=b-c,\quad L_4=2y(2x+a).
 \end{gathered}
\]
Let \(\mu_-\) be the smaller eigenvalue of
\[
 K=\frac1{N^2}
 \begin{pmatrix}
 qV-Q&\sqrt n(qZ-Q)\\
 \sqrt n(qZ-Q)&q(L_1^2+nW)-nQ
 \end{pmatrix}.
\]
Then the exact least singular value is
\[
 \sigma_{\min}(\mathcal M_\phi)^2
 =q\min\left\{\frac q{N^2}
       \min_{1\leq j\leq4}L_j^2,\ \mu_-\right\}.
\]
Thus the unrestricted real three-value trace class admits a reduction to
four scalar expressions and a real symmetric matrix of order two.
\end{theorem}

\begin{proof}
Work first in coefficient space \(\mathbb C^{\mathbb F_q}\), with coordinate
vectors \(e_\lambda\).  Let \(R e_\lambda=e_{-\lambda}\) and
\(T e_\lambda=e_{\lambda+1}+e_{\lambda-1}\).  These commuting self-adjoint
operators preserve the union of any additive \(\mathbb F_3\)-coset and its
negative.  Set
\[
 H_0=\gamma^2I+\overline\gamma^2R
          +a\overline\gamma T+a\gamma TR,\qquad
 w=a^2e_0+e(e_1+e_{-1}).
\]
For \(L=\mathbb F_q\setminus\{1,-1\}\), let \(H\) be the matrix with
columns \(H_0e_\lambda+w\), \(\lambda\in L\).
Expanding \(F(mt)F(m(t+1))\) and using \(\lambda=2t+1\) shows that these
are precisely the additive Fourier coefficient vectors of the unnormalized
columns of \(B_\phi\).  Every column sums to \(F(0)^2=(a+2x)^2\).
Additive-character orthogonality, including removal of \(m=0\), therefore
gives
\[
 B_\phi^*B_\phi=N^{-2}(qH^*H-QJ).
\]

There are \(k=(q-3)/6\) pairs of nonzero cosets in
\(\mathbb F_q/\mathbb F_3\).  On each pair, the joint eigenvalues of
\((T,R)\) are \((2,1),(2,-1),(-1,1),(-1,-1)\), with respective
dimensions \(1,1,2,2\).  The corresponding eigenvalues of \(H_0\) are
\[
 L_1,\qquad iL_2,\qquad L_3,\qquad iL_4.
\]
All but the \((2,1)\) vectors have coordinate sum zero.  Consequently the
normalized Gram eigenvalues on these subspaces are \(qL_j^2/N^2\).
The \((2,1)\) vectors with total coordinate sum zero supply a further
subspace of dimension \(k-1\); the other dimensions are \(k,2k,2k\).
The remaining two-dimensional space is spanned by \(e_0\) and the normalized
constant vector on \(\mathbb F_q\setminus\mathbb F_3\).
Now
\[
 He_0=(a^2+b)e_0+(e+c)(e_1+e_{-1}),
\]
whose squared norm is \(V\), while
\(\|w\|^2=W\) and \(\langle He_0,w\rangle=Z\).
On the normalized constant vector outside \(\mathbb F_3\), \(H\) acts as
\(L_1\) times that vector plus \(\sqrt n\,w\).  Its Gram matrix is
therefore exactly \(K\).

It remains to account for the character block, including the case \(q=9\),
when \(k-1=0\).  The three trace values are
\[
 F_0=a+2x,\quad F_1=a-x-\sqrt3y,\quad F_2=a-x+\sqrt3y.
\]
Lemma~\ref{lem:tracecharacters} gives, for nontrivial \(\chi\),
\[
 |c_\phi(\chi)|^2=
 \begin{cases}
 qL_1^2/N^2,&\chi(-1)=1,\\
 qL_2^2/N^2,&\chi(-1)=-1.
 \end{cases}
\]
Indeed the two numerators are respectively
\(q(2F_0^2-F_1^2-F_2^2)^2/9\) and
\(q(F_1^2-F_2^2)^2/3\).
Both character parities occur for \(q\geq9\), so \(L_1\) remains in the
overall minimum even when its second-block subspace is absent.
Their moduli are at most one, so the trivial character does not change
the minimum.  The exact block identities complete the proof.
\end{proof}

\begin{corollary}[Algebraic recovery criterion for real trace windows]
\label{cor:generaltraceinjective}
With the notation of Theorem~\ref{thm:generaltrace}, put
\[
 \Delta=n(eb-a^2c)+L_1(e+c-a^2-b).
\]
Then
\[
 \det K=\frac{2q\Delta^2}{N^4}.
\]
The orbit generated by \(\phi\) does matrix recovery if and only if
\[
 L_1L_2L_3L_4\Delta\ne0.
\]
\end{corollary}

\begin{proof}
For Fourier coefficient vectors the Gram form, after the sample at zero
is removed, is \(qI-\mathbf1\mathbf1^*\).  It annihilates the constant
vector.  Subtracting the constant value outside \(\mathbb F_3\) from
the second of the two coefficient vectors used to compute \(K\), their
coordinates in \((e_0,e_1+e_{-1})\) become the columns of
\[
 C=\begin{pmatrix}
 a^2+b&(na^2-L_1)/\sqrt n\\
 e+c&(ne-L_1)/\sqrt n
 \end{pmatrix}.
\]
The restricted Gram form is
\[
 G=\begin{pmatrix}q-1&-2\\-2&2q-4\end{pmatrix},
 \qquad \det G=2qn.
\]
Thus \(K=N^{-2}C^*GC\) and
\(\det C=\Delta/\sqrt n\), proving the formula.
Since \(K\) is positive semidefinite, it is positive definite exactly
when \(\Delta\ne0\).  Theorem~\ref{thm:generaltrace} now gives the
stated equivalence.
\end{proof}

\begin{lemma}[A quadratic majorant for the principal sign region]
\label{lem:tracecertificate}
Put \(\ell=\sqrt2-1\), \(\mathcal D_q=q(1-\ell)-1=N_q\), and
\(t_q=\ell/\mathcal D_q\).
Define
\[
 \omega_1=\frac{1+2\ell-(5-2\ell)t_q}{6},\qquad
 \omega_2=\frac{1+t_q}{2},\qquad
 \omega_3=\frac{(1-\ell)(1+t_q)}{3}.
\]
For \(q\geq9\) these weights are strictly positive and sum to one.
For each \(\varepsilon\in\{\pm1\}\), the following identity holds:
\[
\begin{split}
 t_qN-\omega_1L_1-\varepsilon\omega_2L_2-\omega_3L_3
 ={}&t_q(q-1)\bigl(a-(1-\ell)x+\varepsilon\sqrt2\,y\bigr)^2\\
 &+(1+t_q)\bigl(y-\varepsilon\ell x\bigr)^2.
\end{split}
\]
Consequently the left-hand side is nonnegative, and it vanishes
precisely when \((a,x,y)=s(0,1,\varepsilon\ell)\), \(s\in\mathbb R\).
\end{lemma}

\begin{proof}
The relations \(\ell^2=1-2\ell\) and \(\sqrt2\,\ell=1-\ell\)
give the identity by expansion.  The weights sum to one by the same
relations.  Since \(0<t_q\leq t_9<1/10\), they are strictly positive.
Both square coefficients are positive.  Vanishing of the second square
gives \(y=\varepsilon\ell x\), and the first then gives \(a=0\).
\end{proof}

\begin{theorem}[Global sharpness in the principal sign region]
\label{thm:tracechamber}
Under the hypotheses of Theorem~\ref{thm:generaltrace}, suppose that
\(L_1,L_3\geq0\).  Then
\[
 \sigma_{\min}(\mathcal M_\phi)
 \leq \frac{q(\sqrt2-1)}{q(2-\sqrt2)-1}.
\]
Equality holds exactly for the zero-sum optimal trace windows of
Theorem~\ref{thm:traceflat}, up to global sign and interchange of trace
classes \(1\) and \(2\).
\end{theorem}

\begin{proof}
Choose \(\varepsilon\) so that \(\varepsilon L_2=|L_2|\).
The positive weights of Lemma~\ref{lem:tracecertificate} sum to one, so
\[
 \min_{1\leq j\leq4}|L_j|
 \leq \omega_1L_1+\omega_2|L_2|+\omega_3L_3
 \leq t_qN.
\]
Theorem~\ref{thm:generaltrace} gives the bound.
Equality forces \((a,x,y)\) onto one of the two equality lines of
Lemma~\ref{lem:tracecertificate}.  The exact formula in
Theorem~\ref{thm:traceflat} confirms attainment on these lines.
\end{proof}

\begin{lemma}[Certificates for the remaining sign regions]
\label{lem:traceothercertificates}
Let \(z=(a,x,y)^{\mathsf T}\), and write \(L_j=z^{\mathsf T}M_jz\), where
\[
\begin{aligned}
 M_1&=\begin{pmatrix}0&2&0\\2&2&0\\0&0&-2\end{pmatrix},
 &M_2&=\begin{pmatrix}0&0&-2\\0&0&2\\-2&2&0\end{pmatrix},\\
 M_3&=\begin{pmatrix}0&-1&0\\-1&2&0\\0&0&-2\end{pmatrix},
 &M_4&=\begin{pmatrix}0&0&1\\0&0&2\\1&2&0\end{pmatrix}.
\end{aligned}
\]
Put \(E=\operatorname{diag}(1,2,2)\), \(v=(1,2,0)^{\mathsf T}\), and
\(\mathsf N_q=qE-vv^{\mathsf T}\), so that \(N=z^{\mathsf T}\mathsf N_qz\).
For each row in the following table, the weights are nonnegative, sum
to one, and satisfy
\[
 Q_q=t_q\mathsf N_q-\sum_{j=1}^4 w_j s_j M_j\ \succ0
 \qquad(9\leq q<\infty).
\]
Here \(s=(s_1,s_2,s_3,s_4)\) and \(t_q=(\sqrt2-1)/(q(2-\sqrt2)-1)\).
The six rows were obtained by a finite exact search for a convex certificate;
their use below is justified solely by the displayed matrices and direct
algebraic verification, independently of that search.
\[
\begin{array}{c|cccc}
 s&w_1&w_2&w_3&w_4\\ \hline
 (-,+,-,-)&1/2&1/2-\sqrt2/3&0&\sqrt2/3\\
 (-,+,-,+)&1/6&1/6&1/3&1/3\\
 (-,+,+,-)&3/20&1/20&1/4&11/20\\
 (-,+,+,+)&3/10&0&3/20&11/20\\
 (+,+,-,-)&3/20&1/10&3/4&0\\
 (+,+,-,+)&1/10&1/5&11/20&3/20
\end{array}
\]
\end{lemma}

\begin{proof}
Set \(Q_\infty=E/\sqrt2-\sum_j w_js_jM_j\).
Writing \(\mathcal D_q=q(2-\sqrt2)-1\), direct algebra gives
\[
 Q_q=\frac{\mathcal D_9}{\mathcal D_q}\,Q_9+
       \left(1-\frac{\mathcal D_9}{\mathcal D_q}\right)Q_\infty.
\]
It suffices to verify \(Q_9\succ0\) and \(Q_\infty\succeq0\) for the
six displayed rows.  We give explicit principal-minor certificates.
In every row, all three leading principal minors of \(Q_9\) are greater
than \(1/100\).  Its first leading minor is
\((8+64\sqrt2)/127\).  For \(Q_\infty\), the first leading minor is
\(1/\sqrt2\); its second and third leading minors, in the table's order,
are
\[
\begin{array}{c|c}
 \text{second leading minor}&\det Q_\infty\\ \hline
 \sqrt2/2&0\\
 1+\sqrt2/2&0\\
 279/400-\sqrt2/10&739/1000-49\sqrt2/200\\
 7/16+3\sqrt2/20&1971/2000-103\sqrt2/200\\
 -41/400+3\sqrt2/5&1359/1000-353\sqrt2/400\\
 7/16+9\sqrt2/20&57/80-11\sqrt2/40
\end{array}
\]
These values, and the asserted \(Q_9\) bounds, follow by substituting
the displayed matrices and weights; all inequalities can be verified
using \(7/5<\sqrt2<10/7\).
The upper-left block of order two of \(Q_\infty\) is positive definite.
Its Schur complement is nonnegative because its determinant is
nonnegative.  Thus \(Q_\infty\succeq0\).
Sylvester's criterion gives \(Q_9\succ0\), and the coefficient
\(\mathcal D_9/\mathcal D_q\)
is strictly positive for every finite \(q\geq9\), proving the assertion.
\end{proof}

\begin{theorem}[Global optimality over all real three-value trace windows]
\label{thm:traceglobal}
For \(q=3^h\), \(h\geq2\), let \(\mathcal T_q\) consist of all real unit
vectors on \(\mathbb F_q^\times\) which are constant on each absolute-trace
class, with no condition on the sum of their three values.  Then
\[
 \max_{\phi\in\mathcal T_q}\sigma_{\min}(\mathcal M_\phi)
 =\frac{q(\sqrt2-1)}{q(2-\sqrt2)-1}.
\]
Equality holds exactly for the windows of Theorem~\ref{thm:traceflat},
up to global sign and interchange of trace classes \(1\) and \(2\).
In particular, zero trace mean is a consequence of optimality, rather
than a restriction imposed on the optimization problem.
\end{theorem}

\begin{proof}
Use the coordinates of Theorem~\ref{thm:generaltrace}.  If any \(L_j=0\),
the scalar term in its exact formula is zero, and hence the least singular
value is zero, since the remaining term \(\mu_-(K)\) is nonnegative for the
Gram matrix \(K\).  Otherwise interchange trace classes
\(1\) and \(2\), if necessary, to make \(L_2>0\): this sends \(y\) to
\(-y\), preserving \(L_1,L_3,N\) and reversing \(L_2,L_4\).
If \(L_1,L_3>0\), Theorem~\ref{thm:tracechamber} applies, including its
equality classification.

If \(L_1,L_3\) are not both positive, then, after the normalization
\(L_2>0\), their two possible signs together with the sign of \(L_4\)
give exactly the six rows displayed in
Lemma~\ref{lem:traceothercertificates}.  For the corresponding weights,
\[
 \min_j|L_j|\leq\sum_j w_j|L_j|
 =z^{\mathsf T}\left(\sum_j w_js_jM_j\right)z
 <t_q z^{\mathsf T}\mathsf N_qz=t_qN.
\]
The strict inequality follows from \(z\ne0\) and \(Q_q\succ0\).
Theorem~\ref{thm:generaltrace} therefore gives
\(\sigma_{\min}(\mathcal M_\phi)<qt_q\) in every remaining sign region.
Finally, Theorem~\ref{thm:traceflat} supplies attainment of \(qt_q\).
\end{proof}

\begin{corollary}[Sharp least-squares noise amplification]
\label{cor:tracenoise}
Let \(\phi\in\mathcal T_q\) be a matrix-recovery window, equivalently,
let \(\mathcal M_\phi\) be injective.  Suppose that noisy orbit measurements
are given by \(y=\mathcal M_\phi(A)+e\).  The least-squares reconstruction
\(\widehat A=\mathcal M_\phi^\dagger y\) obeys
\[
 \|\widehat A-A\|_F\leq
 \|\mathcal M_\phi^\dagger\|\,\|e\|_2.
\]
Over all matrix-recovery windows \(\phi\in\mathcal T_q\), the smallest
possible worst-case amplification factor is
\[
 \min_{\phi\in\mathcal T_q}\|\mathcal M_\phi^\dagger\|
 =\frac{q(2-\sqrt2)-1}{q(\sqrt2-1)}
 =\sqrt2\left(1-\frac1{q(2-\sqrt2)}\right).
\]
It is attained exactly by the windows in Theorem~\ref{thm:traceflat},
up to their stated symmetries.
\end{corollary}

\begin{proof}
For an injective finite-dimensional measurement operator,
\(\|\mathcal M_\phi^\dagger\|=1/\sigma_{\min}(\mathcal M_\phi)\).
The assertion follows from Theorem~\ref{thm:traceglobal}; its optimizing
windows are matrix-recovery windows by Theorem~\ref{thm:generaltrace}.
Injectivity also gives \(\mathcal M_\phi^\dagger\mathcal M_\phi=I\), so
\(\widehat A-A=\mathcal M_\phi^\dagger e\), which proves the error bound.
\end{proof}

\begin{corollary}[Quantitative certificate stability]
\label{cor:tracecertstability}
Use the notation of Lemma~\ref{lem:tracecertificate}, and normalize the
unnormalized trace function by \(N=1\).  Suppose that \(L_1,L_3\geq0\)
and
\[
 \sigma_{\min}(\mathcal M_\phi)\geq(1-\delta)qt_q,
 \qquad 0\leq\delta<1.
\]
If \(\varepsilon L_2=|L_2|\), then
\[
 t_q(q-1)\bigl(a-(1-\ell)x+\varepsilon\sqrt2\,y\bigr)^2
 +(1+t_q)\bigl(y-\varepsilon\ell x\bigr)^2
 \leq\delta t_q.
\]
In particular,
\[
 \left|a-(1-\ell)x+\varepsilon\sqrt2\,y\right|
 \leq\sqrt{\frac{\delta}{q-1}},
 \qquad
 \left|y-\varepsilon\ell x\right|
 \leq\sqrt{\frac{\delta t_q}{1+t_q}}.
\]
\end{corollary}

\begin{proof}
The exact block formula in Theorem~\ref{thm:generaltrace} implies
\(\min_j|L_j|\geq(1-\delta)t_q\).  The positive weighted average in the
proof of Theorem~\ref{thm:tracechamber} is at least this minimum.
The left-hand side of the claimed inequality is the defect between
\(t_q\) and that average, by Lemma~\ref{lem:tracecertificate}; it is
therefore at most \(\delta t_q\).  Dropping either nonnegative term gives
the two displayed estimates.
\end{proof}

\begin{theorem}[Optimality in the zero-sum trace class]
\label{thm:traceoptimal}
Let \(q=3^h\) with \(h\geq2\), and let \(\mathcal T_q^0\) be the family
of real unit vectors \(\phi\in\mathbb R^{\mathbb F_q^\times}\) of the form
\[
 \phi(m)=u_{\tau(m)},
 \qquad u_0+u_1+u_2=0,
\]
where \(\tau=\operatorname{Tr}_{\mathbb F_q/\mathbb F_3}\).  Then
\[
 \max_{\phi\in\mathcal T_q^0}\sigma_{\min}(\mathcal M_\phi)
 =\frac{q(\sqrt2-1)}{q(2-\sqrt2)-1}.
\]
The maximizers are precisely the trace windows in Theorem~\ref{thm:traceflat},
up to a global sign and interchange of the trace classes \(1\) and \(2\).
\end{theorem}

\begin{proof}
For \(\phi\in\mathcal T_q^0\), write its three unnormalized values as
\[
 F(j)=\gamma\zeta^j+\overline\gamma\zeta^{-j}
 \qquad(j\in\mathbb F_3),
\]
where \(\zeta=e^{2\pi i/3}\) and \(\gamma=\rho e^{i\theta}\ne0\).
Put \(c=\cos^2\theta\).  The squared norm on \(\mathbb F_q^\times\) is
\[
 N=2\rho^2(q-2c).
 \tag{40}
\]
Both invariant subspaces in Lemma~\ref{lem:tracebranches} are nonzero for \(q\geq9\).
Hence the exact block criterion, Lemma~\ref{lem:tracebranches}, and (40)
imply
\[
 \sigma_{\min}(\mathcal M_\phi)
 \leq
 \frac{q}{\sqrt2\,(q-2c)}
 \sqrt{1-|1-8c+8c^2|}.
 \tag{41}
\]

Set \(c_+=(2+\sqrt2)/4\).  If \(0\leq c\leq c_+\), then the square-root
factor in (41) is at most one and \(q-2c\geq q-2c_+\).  Equality in the
resulting bound is possible only at \(c=c_+\).  If \(c_+\leq c\leq1\), the
square-root factor is \(\sqrt{8c(1-c)}\).  The logarithmic derivative of
the right side of (41), apart from its constant factor, is
\[
 \frac{1-2c}{2c(1-c)}+\frac2{q-2c}<0.
\]
Indeed, the magnitude of the first negative summand is at least
\(2\sqrt2\), while the second summand is at most \(2/7\).  Thus the right
side of (41) is maximized uniquely at \(c=c_+\).  At this point it equals
\[
 \frac{q}{\sqrt2(q-2c_+)}
 =\frac{q(\sqrt2-1)}{q(2-\sqrt2)-1}.
\]
The condition \(c=c_+\) is equivalent to \(\operatorname{Re}(\gamma^4)=0\)
with the larger of the two possible values of \(\cos^2\theta\).  After a
global sign and, if necessary, replacing \(\tau\) by \(-\tau\), this is
exactly the window in Theorem~\ref{thm:traceflat}.  That theorem shows the
upper bound is attained, completing the proof.
\end{proof}

\begin{corollary}[Quantitative isolation of the trace optimizer]
\label{cor:traceisolation}
Let \(S_q=q(\sqrt2-1)/(q(2-\sqrt2)-1)\).  For a window
\(\phi\in\mathcal T_q^0\), write its zero-sum parameter as
\(c=\cos^2\theta\) as in the proof of Theorem~\ref{thm:traceoptimal}, and
put \(c_+=(2+\sqrt2)/4\).  If
\[
 \sigma_{\min}(\mathcal M_\phi)\geq(1-\varepsilon)S_q,
 \qquad
 0\leq\varepsilon<\frac{\sqrt2}{2(q-1)},
\]
then
\[
 |c-c_+|\leq\frac{\varepsilon}{\sqrt2}.
 \tag{42}
\]
Thus the unique trace-class optimum is quantitatively stable: a relative
loss of \(\varepsilon\) forces its phase parameter to be within
\(\varepsilon/\sqrt2\) of the optimum.
\end{corollary}

\begin{proof}
Let \(H_q(c)\) denote the right-hand side of (41).  The proof of
Theorem~\ref{thm:traceoptimal} gives
\(\sigma_{\min}(\mathcal M_\phi)\leq H_q(c)\) and
\(H_q(c_+)=S_q\).  If \(c<1/2\), then
\[
 \frac{H_q(c)}{S_q}
 \leq\frac{q-2c_+}{q-1}
 =1-\frac{\sqrt2}{2(q-1)},
\]
which is incompatible with the hypothesis.  Hence \(c\geq1/2\).

Put \(f(c)=1-8c+8c^2\).  For \(1/2\leq c\leq c_+\), we have
\[
 |f(c)|=8(c-c_-)(c_+-c)\geq2\sqrt2(c_+-c),
 \qquad c_-=\frac{2-\sqrt2}{4}.
\]
Using \(\sqrt{1-x}\leq1-x/2\), and observing that
\((q-2c_+)/(q-2c)\leq1\), yields
\[
 \frac{H_q(c)}{S_q}\leq1-\sqrt2(c_+-c).
 \tag{43}
\]
For \(c_+\leq c\leq1\), similarly,
\[
 f(c)=8(c-c_-)(c-c_+)\geq4\sqrt2(c-c_+).
\]
Since
\[
 \frac{q-2c_+}{q-2c}
 \leq1+\frac{2(c-c_+)}{q-2}
 \leq1+\frac{2(c-c_+)}7,
\]
where the final inequality uses \(q\geq9\), hence \(q-2\geq7\).
the same elementary square-root inequality gives
\[
 \frac{H_q(c)}{S_q}
 \leq1-\left(2\sqrt2-\frac27\right)(c-c_+)
 \leq1-\sqrt2(c-c_+).
 \tag{44}
\]
Combining (43)--(44) with
\(H_q(c)\geq\sigma_{\min}(\mathcal M_\phi)\geq(1-\varepsilon)S_q\)
proves (42).
\end{proof}

\begin{corollary}[The trace family beats every two-level window]
\label{cor:tracebeatstwolevel}
For every \(q=3^h\) with \(h\geq2\), the trace window in
Theorem~\ref{thm:traceflat} has a strictly larger least singular value than
every nonnegative two-level window \(\phi_{q,r}\).
\end{corollary}

\begin{proof}
The trace value is larger than \(1/\sqrt2\).  For \(q\geq27\), we have
\(d=q-1\geq10\), and Theorem~\ref{thm:ppoptimal} shows that the optimal
two-level value is
\[
 \frac{\sqrt q}{2(\sqrt{q-1}+1)}<\frac12.
\]
For \(q=9\), direct maximization of the formula in
Theorem~\ref{thm:ppoptimal} gives \(3/7<1/2\).  This proves the claim.
\end{proof}

\begin{remark}[An explicit eight-dimensional example]
\label{thm:q9trace}
Identify \(\mathbb F_9=\mathbb F_3[\omega]/(\omega^2+1)\) and write
\(m=a+b\omega\).  Here \(\operatorname{Tr}(m)=-a\), so using \(a\) instead
of the absolute trace interchanges the two nonzero trace classes.  Thus
Theorem~\ref{thm:traceflat} applies to
\[
 \phi(a+b\omega)=\frac1{\sqrt{17-9\sqrt2}}
 \begin{cases}
 1,&a=0,\\
 r,&a=1,\\
 -1-r,&a=2,
 \end{cases}
 \qquad r=\frac{-1+\sqrt{9-6\sqrt2}}2.
\]
Its exact least singular value is
\[
 \sigma_{\min}(\mathcal M_\phi)
 =\frac{9(\sqrt2-1)}{17-9\sqrt2}
 =\frac{9+72\sqrt2}{127}.
\]
The two block minima equal \((3+24\sqrt2)/127\).  This specializes the
general theorem without any additional finite-field enumeration.
\end{remark}

\begin{theorem}[A universal two-block tradeoff]
\label{thm:universalupper}
Let \(q\) be an odd prime power, \(d=q-1\), and let
\(\phi\in\mathbb R^{\mathbb F_q^\times}\) be a unit vector.  Put
\(X_\phi=\sum_m|\phi(m)|^4\).  Then
\[
\sigma_{\min}(\mathcal M_\phi)^2
\leq q\min\left\{
 \frac{dX_\phi-1}{d-1},\frac{1-X_\phi}{d-1}
 \right\}.
 \tag{45}
\]
Moreover,
\[
 \min\left\{
 \frac{dX_\phi-1}{d-1},\frac{1-X_\phi}{d-1}
 \right\}\leq\frac1q,
 \quad\text{and hence}\quad
 \sigma_{\min}(\mathcal M_\phi)^2\leq1.
\]
Consequently, the optimal two-level orbit in Theorem \ref{thm:ppoptimal}
is within the explicit factor
\[
\frac{\sqrt d+1}{\sqrt q\sin(\pi/(2p))}
 \tag{46}
\]
of the best possible stability among all real unit generating vectors.
For every fixed odd \(p\), this factor is bounded independently of \(h\).
\end{theorem}

\begin{proof}
Let \(a_m=|\phi(m)|^2\), so that \(\sum_ma_m=1\) and
\(\sum_ma_m^2=X_\phi\).  Multiplicative Plancherel gives
\[
 \sum_{\chi\ne1}|c_\phi(\chi)|^2=dX_\phi-1.
\]
It follows that
\[
 \min_\chi|c_\phi(\chi)|^2\leq\frac{dX_\phi-1}{d-1}.
 \tag{47}
\]
The change of variables \((m,t)\mapsto(x,y)=(mt,m(t+1))\) is a bijection
from the index set of \(B_\phi\) to the ordered pairs of distinct elements of
\(\mathbb F_q^\times\).  Therefore
\[
 \|B_\phi\|_F^2=\sum_{x\ne y}a_xa_y=1-X_\phi,
 \qquad
 \sigma_{\min}(B_\phi)^2\leq\frac{1-X_\phi}{d-1}.
 \tag{48}
\]
The exact orthogonal block identities in Proposition \ref{prop:qblock} imply
that the least singular value of the full measurement map is at most \(\sqrt q\)
times either block minimum.  Combining (47)--(48) gives the first inequality
in (45).  The two displayed terms are equal at \(X_\phi=2/(d+1)\), and their
minimum is never larger than \(1/(d+1)=1/q\), proving the second inequality.
Finally, divide this universal upper bound by the value in (31).  Since
\(\sqrt q/(\sqrt d+1)\geq1/2\), the stated factor is also at most
\(2\csc(\pi/(2p))\), independently of \(h\).
\end{proof}

\begin{proposition}[Rigidity at the universal ceiling]
\label{prop:ceilingrigidity}
Let \(\phi\in\mathbb R^{\mathbb F_q^\times}\) be a unit vector.  If
\(\sigma_{\min}(\mathcal M_\phi)=1\), then
\[
 X_\phi=\frac{2}{q},
 \qquad
 |c_\phi(\chi)|=\frac{1}{\sqrt q}
 \quad(\chi\neq\mathbf 1),
 \qquad
 B_\phi^*B_\phi=\frac1q I_{d-1}.
 \tag{49}
\]
\end{proposition}

\begin{proof}
Equality in (45) forces the two scalar terms there to agree, so that
\(X_\phi=2/(d+1)=2/q\).  The character calculation in the proof of
Theorem~\ref{thm:universalupper} then has average
\[
 \frac{1}{d-1}\sum_{\chi\ne\mathbf1}|c_\phi(\chi)|^2=\frac1q.
\]
On the other hand, the exact first-block identity in
Proposition~\ref{prop:qblock} and \(\sigma_{\min}(\mathcal M_\phi)=1\)
give \(|c_\phi(\chi)|^2\geq1/q\) for every nontrivial \(\chi\).  Hence
every one of these inequalities is an equality.  Likewise, the second-block
identity gives \(\sigma_{\min}(B_\phi)^2\geq1/q\), whereas (48) in the
proof of Theorem~\ref{thm:universalupper} gives
\(\|B_\phi\|_F^2=(d-1)/q\).  Thus all \(d-1\) singular values of
\(B_\phi\) equal \(q^{-1/2}\), which is equivalent to the last assertion
in (49).
\end{proof}

\begin{lemma}[The ceiling is tight informational completeness]
\label{lem:ceiling2design}
For a unit vector \(\phi\), let \(P_g=\pi_q(g)\phi\phi^*\pi_q(g)^*\) and
define the frame superoperator
\[
 \mathcal S_\phi(A)=\sum_{g\in G_q}\operatorname{tr}(AP_g)P_g.
\]
Then \(\sigma_{\min}(\mathcal M_\phi)=1\) if and only if
\[
 \mathcal S_\phi(A)=A+\operatorname{tr}(A)I
 \qquad(A\in\mathbb C^{d\times d}).
 \tag{50}
\]
Equivalently, the group-indexed orbit \(\{\pi_q(g)\phi:g\in G_q\}\), with
its uniform group multiplicities, is a tight informationally complete
projective measurement, or, in the usual terminology, a complex projective
\(2\)-design \cite{Scott2006}.
\end{lemma}

\begin{proof}
The superoperator \(\mathcal S_\phi\) is \(\mathcal M_\phi^*\mathcal
M_\phi\).  Irreducibility of \(\pi_q\) and Schur's lemma give
\(\sum_gP_g=qI\), so \(I\) is an eigenvector of \(\mathcal S_\phi\) with
eigenvalue \(q\).  Moreover,
\[
 \operatorname{tr}(\mathcal S_\phi)=\sum_g\|P_g\|_F^2
 =|G_q|=qd.
\]
If \(\sigma_{\min}(\mathcal M_\phi)=1\), the remaining \(d^2-1\)
eigenvalues are at least one, while their sum is
\(qd-q=d^2-1\).  They are therefore all one, which proves (50).
Conversely, (50) gives singular values \(\sqrt q\) on the scalar matrices
and one on their Hilbert--Schmidt orthogonal complement, proving the claim.
The final equivalence is the standard operator form of the projective
\(2\)-design identity.
\end{proof}

\begin{corollary}[Correlation constraints for a ceiling extremizer]
\label{cor:ceilingcorrelation}
Under the hypotheses of Proposition~\ref{prop:ceilingrigidity}, put
\(a(m)=|\phi(m)|^2\) and \(h(m)=\phi(m)\phi(-m)\).  Then, for every
\(r\in\mathbb F_q^\times\),
\[
 \sum_{m\in\mathbb F_q^\times}a(m)a(rm)
 =
 \begin{cases}
  2/q,&r=1,\\
  1/q,&r\ne1.
 \end{cases}
 \tag{51}
\]
Moreover,
\[
 \sum_{m\in\mathbb F_q^\times}h(m)h(rm)=0
 \qquad(r\notin\{1,-1\}).
 \tag{52}
\]
\end{corollary}

\begin{proof}
The multiplicative Fourier transform of the left side of (51) is
\(|c_\phi(\chi)|^2\).  Proposition~\ref{prop:ceilingrigidity}, followed by
Fourier inversion on the group of order \(d=q-1\), gives (51).

For the second claim, reindex the columns of \(B_\phi\) by
\(r=(t+1)/t\in\mathbb F_q^\times\setminus\{1\}\).  The inner product of
the columns indexed by \(r\) and \(r^{-1}\), after setting
\(x=m/(r-1)\), is
\[
 \sum_x\phi(x)\phi(rx)\phi(-rx)\phi(-x)
 =\sum_xh(x)h(rx).
\]
If \(r\notin\{1,-1\}\), these are distinct columns.  The final assertion
of Proposition~\ref{prop:ceilingrigidity} makes their inner product zero.
\end{proof}

\begin{remark}[Normalization benchmark]
\label{rem:benchmark}
The scale in Theorem \ref{thm:primepower} is intrinsic to the unnormalized
coefficient map.  Indeed, for every unit vector \(\phi\), each row functional
of \(\mathcal M_\phi\) is represented in Hilbert--Schmidt space by a rank-one
projector of norm one.  Consequently
\[
 \|\mathcal M_\phi\|_{\mathrm{HS}}^2=|G_q|=q(q-1),
 \qquad
 \sigma_{\min}(\mathcal M_\phi)^2\leq\frac{|G_q|}{d^2}
 =\frac{q}{q-1}.
 \tag{53}
\]
Thus a lower bound bounded away from zero along a fixed-characteristic tower
is the correct dimensional scale, rather than an effect of increasing the
number of measurements.
\end{remark}

\section{Concluding perspective}

The affine orbit provides a setting in which qualitative matrix recovery can
be upgraded to an exact stability theory.  Additive Fourier orthogonality
separates the orbit measurement operator into representation-theoretic
blocks, while multiplicative characters and finite-geometric incidence
structures make the relevant singular values explicit.  This yields both a
sharp two-level design theorem and a uniformly stable three-level family on
the \(3\)-power towers.

The three-value trace calculation goes further: it completely optimizes a
natural finite-field harmonic class, including all equality cases.  Thus the
trace construction is not merely an isolated well-conditioned orbit.  It is
the unique extremizer within a class determined by the absolute-trace
partition.  The universal ceiling and its rigidity conditions give a separate
benchmark for the broader problem of optimizing over arbitrary real windows.

\appendix
\section{Exact verification of the complementary sign certificates}
\label{app:certificates}

This appendix records an independently reproducible exact-arithmetic check of
Lemma~\ref{lem:traceothercertificates}.  It is not a numerical experiment and
is not used in place of the proof: every matrix, weight, and interpolation
identity needed for that proof is displayed in the lemma itself.  Its purpose
is to make the finite list of algebraic inequalities easy to audit.

All entries of \(Q_9\) and \(Q_\infty\) belong to
\(\mathbb Q(\sqrt2)\).  For each of the six sign rows, the accompanying
supplementary file \texttt{audit\_global\_trace\_certificates.py} implements
the field \(\mathbb Q(\sqrt2)\) as pairs of rational numbers and verifies,
using no floating-point operations, all seven nonempty principal minors of
both \(Q_9\) and \(Q_\infty\).  It verifies strict positivity for the seven
minors of \(Q_9\), nonnegativity for the seven minors of \(Q_\infty\), and
the two-square identity of Lemma~\ref{lem:tracecertificate} coefficient by
coefficient.  Thus the six endpoint checks comprise \(84\) exact principal-
minor sign tests.  The elementary bounds
\(7/5<\sqrt2<10/7\) used in the proof are also checked symbolically by
rational comparison.  No floating-point operation or third-party numerical
library is used; a short submission README states the one-command
reproduction procedure.

A second exact script, \texttt{audit\_general\_trace\_exact.py}, audits the
general three-value trace reduction at \(q=9\).  It verifies all \(49\)
entries of the unnormalized second-block Gram matrix and the determinant
identity for \(K\), on an exact \(5\times5\times5\) interpolation grid for
the three real trace values.  The relevant differences have degree at most
four in each variable, so these \(6125\) rational checks certify the stated
polynomial identities at \(q=9\), including the specialization to the
trace-flat window.  As with the first script, this is an independent audit
of the written proof, not a substitute for the general finite-field argument.

For clarity, the interpolation in Lemma~\ref{lem:traceothercertificates} is
an identity in \(\mathbb Q(\sqrt2)\):
\[
 Q_q=\frac{\mathcal D_9}{\mathcal D_q}Q_9+
 \left(1-\frac{\mathcal D_9}{\mathcal D_q}\right)Q_\infty,
 \qquad \mathcal D_q=q(2-\sqrt2)-1.
\]
For every finite \(q\geq9\), its first coefficient is strictly positive,
and the second is nonnegative.  Hence the exact endpoint sign checks,
together with Sylvester's criterion for \(Q_9\), provide a compact
certificate valid uniformly for all \(3\)-power field sizes in the theorem.

\end{document}